\documentclass{amsart}

\usepackage{amsmath, amssymb, amsthm, amsfonts, amsxtra, amscd}

\input xy
\xyoption{all}

\usepackage{hyperref}

\swapnumbers 
\numberwithin{equation}{section}

\theoremstyle{plain}
\newtheorem{theorem}[subsection]{Theorem}
\newtheorem{lemma}[subsection]{Lemma}
\newtheorem{prop}[subsection]{Proposition}
\newtheorem{cor}[subsection]{Corollary}
\newtheorem{conj}[subsection]{Conjecture}

\theoremstyle{definition}
\newtheorem{defn}[subsection]{Definition}

\def\AA{\mathbb{A}}

\def\CC{\mathbb{C}}
\def\DD{\mathbb{D}}

\def\FF{\mathbb{F}}

\def\LL{\mathbb{L}}

\def\PP{\mathbb{P}}
\def\QQ{\mathbb{Q}}

\def\ZZ{\mathbb{Z}}

\def\calA{\mathcal{A}}
\def\calB{\mathcal{B}}
\def\calC{\mathcal{C}}

\def\calE{\mathcal{E}}
\def\calF{\mathcal{F}}

\def\calH{\mathcal{H}}
\def\calI{\mathcal{I}}
\def\calJ{\mathcal{J}}

\def\calL{\mathcal{L}}

\def\calO{\mathcal{O}}

\def\calT{\mathcal{T}}

\def\calV{\mathcal{V}}

\def\calY{\mathcal{Y}}

\def\bR{\mathbf{R}}

\newcommand{\tpi}{\widetilde{\pi}}
\newcommand{\tilL}{\widetilde{L}}
\newcommand{\tilT}{\widetilde{T}}
\newcommand{\tilU}{\widetilde{U}}
\newcommand{\tils}{\widetilde{s}}

\newcommand\hatO{\widehat{\calO}}
\newcommand\hatK{\widehat{K}}
\newcommand\hatR{\widehat{R}}

\newcommand{\tcB}{\widetilde{\calB}}
\newcommand{\tcC}{\widetilde{\calC}}

\renewcommand{\frm}{\mathfrak{m}}

\newcommand\una{\underline{a}}

\newcommand\bareta{\overline{\eta}}

\newcommand{\codim}{\textup{codim}}

\newcommand{\Frob}{\textup{Frob}}
\newcommand\Gal{\textup{Gal}}
\newcommand{\Gr}{\textup{Gr}}
\newcommand{\Hilb}{\textup{Hilb}}
\newcommand\id{\textup{id}}
\newcommand\Jac{\textup{Jac}}
\newcommand\mult{\textup{mult}}

\newcommand{\Pic}{\textup{Pic}}

\newcommand\rank{\textup{rank}}
\newcommand{\red}{\textup{red}}

\newcommand\rk{\textup{rk}}

\newcommand\sing{\textup{sing}}
\newcommand\Spec{\textup{Spec}}
\newcommand\Spf{\textup{Spf}}

\newcommand\Sym{\textup{Sym}}
\newcommand\Tot{\textup{Tot}}
\newcommand{\univ}{\textup{univ}}
\newcommand{\val}{\textup{val}}
\newcommand{\Var}{\textup{Var}}
\newcommand{\vir}{\textup{vir}}

\newcommand\Hom{\textup{Hom}}

\newcommand\uHom{\underline{\Hom}}

\newcommand{\Ext}{\textup{Ext}}

\newcommand\GL{\textup{GL}}

\newcommand{\isom}{\stackrel{\sim}{\to}}

\newcommand{\leftexp}[2]{{\vphantom{#2}}^{#1}{#2}}
\newcommand{\pH}{\leftexp{p}{\textup{H}}}

\newcommand{\Ql}{\QQ_\ell}
\newcommand{\const}[1]{\QQ_{\ell,#1}}
\newcommand{\twtimes}[1]{\stackrel{#1}{\times}}
\newcommand{\cohog}[2]{\textup{H}^{#1}({#2})}     % plain group
\newcommand{\HBM}[2]{\textup{H}^{\textup{BM}}_{#1}({#2})}  % Borel Moore homology

\newcommand{\cJac}{\overline{\textup{Jac}}}
\newcommand{\cJ}{\overline{\calJ}}
\newcommand{\oleft}{\overleftarrow}
\newcommand{\oright}{\overrightarrow}
\newcommand{\sm}{\textup{sm}}
\newcommand{\hotimes}{\widehat{\otimes}}

\newcommand{\pr}{\textup{pr}}
\newcommand{\cPic}{\overline{\Pic}}
\newcommand{\Aint}{\calA^{\textup{int}}}

\title{Macdonald formula for curves with planar singularities}
\author{Davesh Maulik}
\address{MIT, 77 Massachusetts Avenue, Cambridge MA 02139}
\email{dmaulik@math.mit.edu}
\author{Zhiwei Yun}
\address{MIT, 77 Massachusetts Avenue, Cambridge MA 02139}
\email{zyun@math.mit.edu}
\date{}

\begin{document}

\begin{abstract}
We generalize Macdonald's formula for the cohomology of Hilbert schemes of points on a curve from smooth curves to curves with planar singularities: we relate the cohomology of the Hilbert schemes to the cohomology of the compactified Jacobian of the curve. The new formula is a consequence of a stronger identity between certain perverse sheaves defined by a family of curves satisfying mild conditions. The proof makes an essential use of Ng\^o's support theorem for compactified Jacobians and generalizes this theorem to the relative Hilbert scheme of such families. As a consequence, we generalize part of the Weil conjectures to the Hilbert-zeta function of curves with planar singularities.

\end{abstract}

\maketitle

\section{Introduction}
Let $C$ be a smooth projective connected curve over an algebraically closed field $k$. Let $\Sym^n(C)$ be the $n$-th symmetric product of $C$. Macdonald's formula \cite{Mac} says there is a canonical isomorphism between graded vector spaces
\begin{equation}\label{Mac}
\cohog{*}{\Sym^n(C)}\cong\Sym^n(\cohog{*}{C})=\bigoplus_{i+j\leq n, i,j\geq0}\bigwedge^i(\cohog{1}{C})[-i-2j](-j).
\end{equation}
Here $[?]$ denotes cohomological shift and $(?)$ denotes the Tate twist, and this formula respects weight filtrations (when $k\subset\CC$) or Frobenius actions (when $C$ is defined over a finite field). If we define the cohomological zeta function of $C$ to be the formal power series in one variable $t$ with coefficients in graded vector spaces:
\begin{equation}\label{L}
Z(t,C):=\frac{\bigoplus_{i=0}^{2g}(\bigwedge^i\cohog{1}{C})[-i]t^i}{(1-\Ql t)(1-\Ql[-2](-1)t)}
\end{equation}
we may rewrite \eqref{Mac} for all $n$ at once as an identity between formal power series
\begin{equation*}
\bigoplus_{n\geq0}\cohog{*}{\Sym^n(C)}t^n\cong Z(t,C).
\end{equation*}

The purpose of this note is to generalize Macdonald's formula to projective integral curves $C$ with planar singularities. In this case, we will work with the Hilbert schemes $\Hilb^n(C)$ instead of symmetric powers. We can reinterpret the numerator of \eqref{L} as the total cohomology of the Jacobian $\Jac(C)$ of $C$. In the case $C$ is singular,  we work with the compactified Jacobian $\cJac(C)$ classifying torsion-free, rank one coherent sheaves with a fixed degree on $C$.

The main result of this note is the following theorem, which was conjectured by L.Migliorini. While writing this note, we learned that L.Migliorini and V.Shende have an independent proof of the conjecture in the case $k=\CC$.

\begin{theorem}\label{th:main}
Let $C$ be an integral projective curve over $k$ of arithmetic genus $g_{a}$ with planar singularities. Assume either $\textup{char}(k)=0$ or $\textup{char}(k)>\max\{\mult_{p}(C);p\in C\}$ ($\mult_{p}(C)$ is the multiplicity of $C$ at $p$). Then there exists a canonical increasing filtration $P_{\leq i}$ on $\cohog{*}{\cJac(C)}$, normalized such that $\Gr^{P}_{i}\cohog{*}{\cJac(C)}=0$ unless $0\leq i\leq 2g_{a}$, such that for each non-negative integer $n$, there is an isomorphism between {\em graded} vector spaces (gradings are given by $*$, and are shifted by $[-2j]$ in the usual way)
\begin{equation}\label{main}
\cohog{*}{\Hilb^n(C)}\cong\bigoplus_{i+j\leq n, i,j\geq0}\Gr^{P}_{i}(\cohog{*}{\cJac(C)})[-2j](-j).
\end{equation}

\end{theorem}
The canonical filtration $P_{\leq i}$ on $\cohog{*}{\cJac(C)}$ is defined by embedding the curve $C$ into a suitable family of curves, and taking the perverse filtration on the direct image complex of the compactified Jacobian of the family. For this reason, we call $P_{\leq i}$ the {perverse filtration}. It turns out that this filtration is independent of the choice of the family (see Proposition \ref{p:canonfil} and \S\ref{ss:proof}).

If we define the cohomological zeta function of $C$ with respect to the perverse filtration to be
\begin{equation*}
Z_{P}(t,C):=\frac{\bigoplus_{i=0}^{2g_{a}}\Gr^{P}_{i}\cohog{*}{\cJac(C)}t^i}{(1-\Ql t)(1-\Ql[-2](-1)t)}.
\end{equation*}
Then we can restate the main theorem as an identity between formal power series in graded vector spaces
\begin{equation}\label{Zeta}
\bigoplus_{n\geq0}\cohog{*}{\Hilb^n(C)}t^n\cong Z_{P}(t,C).
\end{equation}

Qualitatively, this theorem says that the cohomological information of all the Hilbert schemes is already contained in the cohomology of the compactified Jacobian, equipped with the perverse filtration $P_{\leq i}$.  Notice that, in this expression, the grading by number of points on the left-hand side is partially converted into the perverse filtration.

When the relevant cohomology groups carry weight filtrations (if $k\subset\CC$) or Frobenius actions (if $C$ is defined over a finite field), the isomorphism \eqref{main} respects these extra structures. When $C$ is smooth, we recover \eqref{Mac} from \eqref{main}. For a curve $C$ defined over a finite field $\FF_{q}$, one may consider its Hilbert-zeta function:
\begin{equation*}
Z_{\Hilb}(t,C/\FF_{q})=\sum_{n\geq0}\#\Hilb^{n}(C)(\FF_{q})t^{n}.
\end{equation*}
Our main result has the following consequence which generalizes part of the Weil conjecture for smooth curves.

\begin{theorem}\label{th:zeta} Let $C$ be an integral projective curve of arithmetic genus $g_{a}$ over $\FF_{q}$ with planar singularities. Suppose $\textup{char}(\FF_{q})>\max\{\mult_{p}(C);p\in C(\overline{\FF}_{q})\}$. Then
\begin{enumerate}
\item $Z_{\Hilb}(t,C/\FF_{q})$ has the form $\frac{P(t)}{(1-t)(1-qt)}$ for some polynomial $P(t)\in\ZZ[t]$ of degree $2g_{a}$.
\item $Z_{\Hilb}(t,C/\FF_{q})$ satisfies the functional equation
\begin{equation*}
Z_{\Hilb}(t,C/\FF_{q})=(qt^{2})^{(g_{a}-1)}Z_{\Hilb}(q^{-1}t^{-1},C/\FF_{q}).
\end{equation*}
\end{enumerate}
\end{theorem}

We also have a local version of the above results, with $C$ replaced by the completed local ring $\hatO$ of a planar curve singularity. The local version is weaker than Theorem \ref{th:main} in the sense that it is an identity between virtual Poincar\'e polynomials instead of the cohomology groups themselves. For more details, see \S\ref{ss:local} and Theorem \ref{th:local}.

In the local case, there is a conjectural relation between the virtual Poincar\'e polynomial of the punctual Hilbert scheme and the Khovanov-Rozansky homology of the associated link of the singular point, proposed by Oblomkov, Rasmussen and Shende (\cite{ORS} and \cite{OS}).  It would be interesting to see whether this relation can be better understood in terms of the Lefschetz filtration on the cohomology of the compactified Jacobian.

\subsection{Idea of the proof}
Theorem \ref{th:main} is proved by embedding $C$ into a suitable family of curves, satisfying three conditions axiomatized as (A1)-(A3) in \S\ref{ss:assumption}. For any such family of curves $\calC\to\calB$, we prove a global analog of the formula \eqref{main} in Theorem \ref{th:pervmac}, which is an identity because perverse sheaves on the base of the family coming from the cohomology of the relative Hilbert scheme and the relative compactified Jacobian. The key step in the proof of the global formula is Proposition \ref{p:supp} saying that any simple perverse constituent of the direct image complex of $\Hilb^{n}(\calC/\calB)\to\calB$ has full support $\calB$. We use a cohomological correspondence argument for a descending induction on the number of points, reducing this support theorem to Ng\^o's support theorem for compactified Jacobians (Theorem \ref{th:supp}). Therefore, our result is an application of Ng\^o's powerful technique to a more classical setting.

\noindent{\bf Acknowledgement}
The authors would like to thank A. Oblomkov for bringing this problem to their attention and V. Shende for helpful discussions. D.M. is supported by the Clay Research Fellowship. Z.Y. is partially supported by the NSF grant DMS-0969470.

\section{Macdonald formula for families of curves}

\subsection{Assumptions on the family of curves}\label{ss:assumption}
In this section, $k$ is any algebraically closed field.
Let $\pi: \calC\to \calB$ be a locally projective flat family of curves over an irreducible base $\calB$. Let $g_{a}$ be the common arithmetic genus of the fibers of $\pi$. For each integer $n\geq0$, let $f_n:\calH_n=\Hilb^{n}(\calC/\calB)\to \calB$ be the relative Hilbert scheme of $n$-points on the fibers of $\pi$. We assume:
\begin{enumerate} 
\item[(A1)] Each geometric fiber $\calC_b$ ($b\in \calB$) is integral and has only planar singularities;
\item[(A2)] For each $0\leq n\leq2g_{a}-1$, the total space $\calH_n$ is smooth;
\item[(A3)] For every (not necessarily closed) point $b\in\calB$, we have the $\delta$-invariant $\delta(b)$ of the fiber $\calC_b$. Then we have
\begin{equation*}
\codim_{\calB}(\overline{\{b\}})\geq\delta(b).
\end{equation*}
Here $\overline{\{b\}}$ is the Zariski closure of $b$ in $\calB$.
\end{enumerate}
Note that (A2) for $n=0$ implies $\calB$ is smooth. We denote $\dim\calB$ by $d_{\calB}$. Also, (A3) implies that the generic fiber of $\pi$ is smooth, because the locus with $\delta(b)\geq1$ has codimension at least one in $\calB$.

\subsection{Compactified Jacobians}\label{ss:Jac}
In the following discussion, we assume
\begin{enumerate}
\item[(A4)] The family $\pi$ admits a section $s:\calB\to\calC^{\sm}$, where $\calC^{\sm}\subset\calC$ is the open subscheme with $\calC_{b}^{\sm}$ being the smooth locus of $\calC_{b}$ for each $b\in\calB$.
\end{enumerate}
With this assumption, one may define compactified Picard schemes $p_n:\cPic_n\to \calB$ of the family $\pi:\calC\to \calB$. More precisely, $\cPic_n$ is the sheafification of the following presheaf: it sends every commutative $k$-algebra $R$ to the set $\cPic_n(R)$, the set of isomorphism classes of triples $(b,\calF,\tau)$, where $b\in \calB(R)$, $\calF$ is a torsion-free, generically rank one coherent sheaf on $\calC_b$ (a curve over $\Spec R$), flat over $R$, with 
\begin{equation*}
\chi(\calC_{b'},\calO_{\calC_{b'}})-\chi(\calC_{b'},\calF_{b'})=n
\end{equation*}
for every geometric point $b'\in\Spec R$, and $\tau$ is an isomorphism of $R$-modules $\calF_{s(b)}\isom R$.

We can use the section $s$ to identified the various components $\cPic_n$. After this identification, we denote it by $\cJ$, the relative compactified Jacobian.

\begin{lemma}\label{l:Jsmooth} Under (A1)-(A2) and (A4), we have
\begin{enumerate}
\item The relative compactified Jacobian $\cJ$ is smooth over $k$.
\item Each geometric fiber $\cJ_b$ of $\cJ$ is irreducible of dimension $g_{a}$.
\end{enumerate}
\end{lemma}
\begin{proof}
By a result of Altman and Kleiman \cite[Theorem 8.4(v)]{AK}, the morphism $\calH_{n}\to\cPic_n$ sending a subscheme of $\calC$ to its ideal sheaf is a projective space bundle for $n\geq2g_{a}-1$. Since $\calH_{2g_{a}-1}$ is smooth by (A2), so is $\cPic_{2g_{a}-1}$, hence $\cJ$. (2) is the main result of \cite{AIK}.
\end{proof}

Let $p:\cJ\to\calB$ be the projection and $L=\bR p_*\const{\cJ}\in D^{b}_{c}(\calB)$. Since $\cJ$ is smooth and $p$ is locally projective, by the Decomposition Theorem \cite[Th\'eor\`eme 6.2.5]{BBD}, we have (non-canonically)
\begin{equation}\label{decompL}
L\cong\bigoplus_i\pH^iL[-i]
\end{equation}
as objects in $D^b_c(\calB)$. Each perverse sheaf $\pH^iL$ is semisimple.

\begin{theorem}[{B.C.Ng\^o \cite[Th\'eor\`eme 7.2.1]{NgoFL}}]\label{th:supp}
Assume (A1)-(A4) hold. For every $i\in\ZZ$ and every simple constituent $M$ of $\pH^iL$, the support of $M$ is the whole of $\calB$.
\end{theorem}
\begin{proof}
We will apply the general result of Ng\^o \cite[Th\'eor\`eme 7.2.1]{NgoFL}. To this end, we need to check that the action of the relative Jacobian $\Jac(\calC/\calB)$ over $\calB$ on $\cJ$ satisfies the assumptions in \cite[\S7.1]{NgoFL}. First of all, $\cJ\to\calB$ is locally projective by Altman and Kleiman \cite{AK}. \cite[7.1.2]{NgoFL} follows from Lemma \ref{l:Jsmooth}(2); \cite[7.1.3]{NgoFL} is checked similarly as in \cite[Corollaire 4.15.2]{NgoFL}, using the analog of the product formula \cite[Proposition 4.15.1]{NgoFL} in the setting of the $\Jac(C)$-action on $\cJac(C)$ (for details see \eqref{prod}); \cite[7.1.4]{NgoFL} is checked in \cite[\S4.12]{NgoFL} using the Weil pairing; \cite[7.1.5]{NgoFL} follows from (A3). Therefore  \cite[Th\'eor\`eme 7.2.1]{NgoFL} is applicable.

Let $Z$ the support of $M$. The conclusion of \cite[Th\'eor\`eme 7.2.1]{NgoFL} is that $M$ appears as a direct summand of $\bR^{2g_{a}}p_!\Ql$, at least when restricted to an \'etale neighborhood of the generic point of $Z$. But by Lemma \ref{l:Jsmooth}(2), $\bR^{2g_{a}}p_!\Ql=\Ql$, which is an irreducible perverse sheaf up to shift, because $\calB$ is smooth. Hence $Z=\calB$.
\end{proof}

\subsection{What if there is no section?}\label{ss:nosec} In this subsection, we discuss why the perverse sheaves $\pH^{i}L$ and Theorem \ref{th:supp} still make sense even if we drop the assumption (A4).

Let $\varphi:\tcB\to\calB$ be an \'etale surjective morphism, and $\tils:\tcB\to\calC^{\sm}$ be a morphism such that $\pi\circ\tils=\varphi$. Such a pair $(\tcB,\tils)$ always exists.

Let $\tcC=\tcB\times_{\calB}\calC$. Then $\tpi:\tcC\to\tcB$ satisfies the assumptions (A1)-(A4). The above discussion on compactified Jacobians make sense for the family $\tpi$. In particular, we have the direct image complex $\tilL\in D^{b}_{c}(\tcB)$ of the compactified Jacobian for $\tpi$, and the perverse sheaves $\pH^{i}\tilL$ satisfy the property of Theorem \ref{th:supp}. The perverse sheaves $\pH^{i}\tilL$ carry obvious descent data with respect to the \'etale covering $\tcB\to\calB$ (here we use the fact that the compactified Jacobian, once representable, is independent of the section of the family). By \cite[2.2.19]{BBD}, perverse sheaves satisfy \'etale descent, there exist (unique up to unique isomorphism) perverse sheaves $L^{i}$ on $\calB$ such that $\varphi^{*}L^{i}\cong\pH^{i}\tilL$. Clearly, Theorem \ref{th:supp} holds for $L^{i}$ also. It is also easy to see that $L^{i}$ is independent of the choice of $(\tcB,\tils)$.

In the following we will use these perverse sheaves $L^{i}$ when $\pH^{i}L$ is not defined.

\begin{lemma}\label{l:Hsmooth} Under the assumptions (A1)-(A3), we have
\begin{enumerate}
\item The relative Hilbert scheme $\calH_{n}$ is smooth and irreducible for all $n\geq0$;
\item The morphism $f_{n}:\calH_{n}\to\calB$ is flat, with all geometric fibers irreducible of dimension $n$.
\end{enumerate}
\end{lemma}
\begin{proof}
The statements in the lemma are local for the \'etale topology, so we may assume (A4) also holds, hence the compactified Jacobian $\cJ$ exists. The argument and the statement of Lemma \ref{l:Jsmooth}(1) shows that $\calH_{n}$ is smooth over $\cJ$ for $n\geq2g_{a}-1$, hence smooth. For $n\leq2g_{a}-1$, $\calH_{n}$ is smooth by (A1). Hence $\calH_{n}$ is smooth for all $n$.

The rest of the lemma will follow once we show that every geometric fiber $\Hilb^{n}(\calC_{b})$ of $f_{n}$ is irreducible of dimension $n$, for all $n$. 

Let $C=\calC_{b}$, which we may assume to be embedded into a smooth irreducible surface $S$. We first show that every component of $\Hilb^{n}(C)$ has dimension at least $n$. In fact, consider the vector bundle $\calE$ over $\Hilb^{n}(S)$ whose fiber at $Z\in\Hilb^{n}(S)$ is $\Hom_{S}(\calO_{S}(-C),\calO_{Z})$. Then $\calE$ carries a canonical section $s$ given by $\calO_{S}(-C)\subset\calO_{S}\twoheadrightarrow\calO_{Z}$, and $\Hilb^{n}(C)$ is cut out by the vanishing of $s$. Since $\Hilb^{n}(S)$ is smooth of dimension $2n$ and $\rank(\calE)=n$, every component of $\Hilb^{n}(C)$ has dimension at least $n$. It is also clear that the closure of $\Hilb^{n}(C^{\sm})$ in $\Hilb^{n}(C)$ is a component of dimension $n$.

For $n$ large, $\Hilb^{n}(C)$ is a projective space bundle over $\cJac(C)$, which is irreducible by Lemma \ref{l:Jsmooth}(2). Therefore $\Hilb^{n}(C)$ is irreducible of dimension $n$ for large $n$. Suppose $n_{0}$ is the largest number for which $\Hilb^{n_{0}}(C)$ has an extra component $H'$. Consider the quasi-finite map $a:\Hilb^{n}(C)\times C^{\sm}\to\Hilb^{n+1}(C)$ of ``adding one smooth point''. The generic point of $H'$ has support involving singular points of $C$, while the generic point of $\Hilb^{n+1}(C)$ has support in $C^{\sm}$, therefore $a(H'\times C^{\sm})$ is not dense in $\Hilb^{n+1}(C)$. Since $a$ is quasi-finite, $\dim(H'\times C^{\sm})<n+1$, i.e., $\dim(H')<n$. This contradict the lower bound for all components of $\Hilb^{n}(C)$. Therefore $\Hilb^{n}(C)$ is irreducible of dimension $n$.
\end{proof}

\subsection{The Hilbert-Chow map}
For each $b\in\calB$, there is a Zariski open neighborhood $\calB'$ of $b$ over which we can
arrange an embedding $\calC|_{\calB'}\subset \mathbb{P}^N_{\calB'}$. For a generic
choice of linear projection, the induced map $\calC_{b}\rightarrow
\PP^1$ is finite, and will remain so for $\calC_{b'}$ for $b' \in \calB''$, where $\calB''\subset\calB'$ is another some Zariski open neighborhood of $b$. Therefore, Zariski locally on $\calB$, we may make the following assumption:
\begin{enumerate}
\item[(A5)] There is a smooth, connected, projective curve $X$ over $k$ and a finite morphism $\nu:\calC\to X\times\calB$ lifting $\pi$.
\end{enumerate} 
We will assume (A5) in the following discussion.

Consider the morphism
\begin{equation*}
\nu_n:\calH_n\to\Sym^n(\calC/\calB)\to\Sym^n(X)\times\calB.
\end{equation*}
where the first arrow is the Hilbert-Chow map relative to the base $\calB$, and the second map is induced from the finite morphism $\nu:\calC\to X\times\calB$.
which is proper. We understand $\Sym^{0}(X)$ as $\Spec k$.

We recall from \cite[\S 6.2]{GM} the notion of a small map.
\begin{defn}
A proper surjective morphism $f:Y\to X$ between {\em irreducible} schemes over $k$ is called {\em small} if for any $d\geq1$, we have
\begin{equation}\label{eq:cosm}
\codim_{X}\{x\in X|\dim f^{-1}(x)\geq d\}\geq 2d+1.
\end{equation}
\end{defn}
 
\begin{lemma}\label{l:small}
For each $n\in\ZZ_{\geq0}$, the morphism $\nu_{n}$ is small.
\end{lemma}
\begin{proof}
By Lemma \ref{l:Hsmooth}(1), $\calH_{n}$ is irreducible. The morphism $\nu_{n}$ is clearly proper and surjective.

Fix an integer $d\geq1$. Let $Z_{d}\subset\Sym^{n}(X)\times\calB$ be the closed locus where the fibers of $\nu_{n}$ have dimension at least $d$. Let $\zeta=(D=n_{1}x_{1}+\cdots+n_{r}x_{r},b)$ be a point of $Z_{i}$, with values in some field $F$, over a generic point of $Z_{i}$. Here $x_{i}\in X(F)$ are distinct points and $n_{i}>0$. Consider the Hilbert-Chow map of $\mu_{m}:\Hilb^{m}(\calC_{b})\to\Sym^{m}(\calC_{b})$. Let $y\in\calC_{b}$ be any closed point. Since $\calC_{b}$ is locally planar, $\dim\mu^{-1}_{m}(my)\leq m-1$ by a result of Iarrobino \cite[Corollary 2]{Iarrobino}. On the other hand, $\mu^{-1}_{m}(my)$ classifies length $m$ quotients of $\hatO_{\calC_{b},y}$, which is a subscheme of $\cPic(\hatO_{\calC_{b},y})$, to be defined in \S\ref{ss:local}. Since $\cPic(\hatO_{\calC_{b},y})$ has dimension $\delta(\calC_{b};y)$ (the local $\delta$-invariant of $\calC_{b}$ at $y$), we have $\dim\mu^{-1}_{m}(my)\leq\delta(\calC_{b};y)$. Summarizing,
\begin{equation*}
\dim\mu_{m}^{-1}(my)\leq\min\{m-1,\delta(\calC_{b};y)\}, \textup{for all closed points }y\in \calC_{b}, m\geq1.
\end{equation*}
This implies that
\begin{equation}\label{ineq}
d=\dim\nu_{n}^{-1}(D,b)\leq\sum_{i=1}^{r}\min\{n_{i}-1,\delta(\calC_{b};x_{i})\}.
\end{equation}
Here $\delta(\calC_{b};x)=\sum_{y\in\nu^{-1}(x,b)}\delta(\calC_{b};y)$ for $x\in X(F)$. In particular, $d\leq\delta(\calC_{b})$. Hence, by (A3), 
\begin{equation}\label{codimB}
\codim_{\calB}(\overline{\{b\}})\geq \delta(\calC_{b})\geq d.
\end{equation}

Let $S\subset X\otimes_{k}F$ be the finite subscheme consisting of those $x$ with $\delta(b;x)>0$. The inequality \eqref{ineq} implies that at least $d+1$ of the points in $D$ (counted with multiplicities) which are from $S$. This implies that
\begin{equation}\label{codimX}
\codim_{\Sym^{n}(X)}(\overline{\{D\}})\geq d+1.
\end{equation}
Adding \eqref{codimB} and \eqref{codimX} together we get
\begin{equation*}
\codim_{\Sym^{n}(X)\times\calB}(\overline{\{\zeta\}})\geq 2d+1.
\end{equation*}
This being true for all generic points of $Z_{d}$, we conclude that
\begin{equation*}
\codim_{\Sym^{n}(X)\times\calB}(Z_{d})\geq 2d+1.
\end{equation*}
This verifies the smallness of $\nu_{n}$.
\end{proof}

Let $E_{n}=\bR\nu_{n,*}\Ql\in D^{b}_{c}(\Sym^{n}(X)\times\calB)$. In particular, $E_{0}$ is the constant sheaf on $\calB$. Also $E_{1}=\nu_{*}\Ql$ is a sheaf on $X\times\calB$, and $E_{1}[d_{\calB}+1]$ is a perverse sheaf since $\nu$ is finite.

Let $U_{n}\subset\Sym^{n}(X)$ be the open subscheme consisting of multiplicity-free divisors. Let $\tilU_{n}\subset X^{n}$ be the preimage of $U_{n}$, which is an $S_{n}$-torsor over $U_{n}$. The sheaf $E_{1}^{\boxtimes n}|_{\tilU_{n}\times\calB}$ on $\tilU_{n}\times\calB$ admits an obvious $S_{n}$-equivariant structure and hence descends to a sheaf $\Sym^{n}(E_{1})$ on $U_{n}\times\calB$.

\begin{cor}\label{c:midext}
The complex $E_n[d_{\calB}+n]$ is a perverse sheaf on $\Sym^{n}(X)\times\calB$, and we have a canonical isomorphism
\begin{equation*}
E_{n}[d_{\calB}+n]\cong j_{n,!*}(\Sym^{n}(E_{1})[d_{\calB}+n]).
\end{equation*}
Here $j_{n}:U_{n}\times\calB\hookrightarrow\Sym^{n}(X)\times\calB$ is the open inclusion, and we are implicitly stating that $\Sym^{n}(E_{1})[d_{\calB}+n]$ is a perverse sheaf on $U_{n}\times\calB$.
\end{cor}
\begin{proof}
By Lemma \ref{l:small}, the morphism $\nu_{n}$ is small. Since $\calH_{n}$ is smooth for all $n$ by Lemma \ref{l:Hsmooth}(1), the complex $\bR\nu_{n,*}\Ql[d_{\calB}+n]=E_{n}[d_{\calB}+n]$ is perverse, and is the middle extension of its restriction to any open dense subset (in particular $U_{n}\times\calB$) of $\Sym^{n}(X)\times\calB$. Clearly $j_{n}^{*}E_{n}=\Sym^{n}(E_{1})$, hence $E_{n}[d_{\calB}+n]$ is the middle extension of the perverse sheaf $\Sym^{n}(E_{1})[d_{\calB}+n]$. 
\end{proof}

\subsection{The shift operator}
Consider the following diagram
\begin{equation}\label{T}
\xymatrix{& \calT_n\ar[dl]_{\oleft{t}}\ar[dr]^{\oright{t}} \\
\calH_n\times X\ar[d]^{\nu_{n}\times\id_{X}} & & \calH_{n+1}\ar[d]^{\nu_{n+1}}\\
\Sym^{n}(X)\times X\times \calB\ar[rr]^{\sigma_{n}} & & \Sym^{n+1}(X)\times\calB}
\end{equation}
For any $k$-algebra $R$, $\calT_n(R)$ classifies the data $(b,x,\calI'\subset\calI\subset\calO_{\calC_b})$ where $(b,x)\in \calB(R)\times X(R)$ such that $\calI/\calI'$ is an invertible $R$-module supported over the graph of $x$. The morphism $\oleft{t}$ (resp. $\oright{t}$) sends $(b,x,\calI'\subset\calI)$ to $(b,x,\calI')$ (resp. $(b,x,\calI)$). The morphism $\sigma_{n}$ sends $(D,x,b)$ to $(D+x,b)$.

The restriction of $\calT_{n}$ to the generic point $\eta\in\calB$ can be identified with
\begin{equation}\label{genericT}
\xymatrix{& \Sym^{n}(\calC_{\eta})\times\calC_{\eta}\ar[dl]_{\oleft{t}}\ar[dr]^{\oright{t}} \\
\Sym^{n}(\calC_{\eta})\times X & & \Sym^{n+1}(\calC_{\eta})}
\end{equation}

Let $[\calT_{n}]\in\HBM{2(d_{\calB}+n+1)}{\calT_{n}}=\cohog{0}{\calT_{n},\oright{t}^{!}\const{\calH_{n+1}}}$ be the class of the closure of the generic fiber $\calT_{n,b}$. Here we used the fact that $\calH_{n+1}$ of dimension $d_{\calB}+n+1$ to identify $\DD_{\calT_{n}}[-2(d_{\calB}+n+1)](-d_{\calB}-n-1)$ with $\oright{t}^{!}\const{\calH_{n+1}}$. By the formalism of cohomological correspondences \cite{SGA5} or \cite[Appendix A.1]{GS}, $[\calT_{n}]$ can be viewed as a cohomological correspondence between the constant sheaves on $\calH_{n}\times X$ and $\calH_{n+1}$ with support on $\calT_{n}$:
\begin{equation*}
[\calT_{n}]:\oleft{t}^{*}\const{\calH_{n}\times X}\to\oright{t}^{!}\const{\calH_{n+1}}.
\end{equation*}
It induces a map
\begin{equation*}
\tilT_n=[\calT_{n}]_{\#}:\sigma_{n,*}(E_n\boxtimes\const{X})\to E_{n+1}
\end{equation*}
between shifted perverse sheaves on $\Sym^{n+1}(X)\times\calB$. Let
\begin{equation*}
K_n=\bR f_{n,*}\Ql\in D^b_c(\calB).
\end{equation*}
Taking the direct image of $\tilT_{n}$ under $\bR f_{n+1,*}$, we get
\begin{equation*}
T_{n}=\bR f_{n+1,*}(\tilT_{n}):K_{n}\otimes\cohog{*}{X}=\bR f_{n+1,*}\sigma_{n,*}(E_{n}\boxtimes\const{X})\to\bR f_{n+1,*}E_{n+1}=K_{n+1}
\end{equation*}

\begin{prop}\label{p:supp} Under (A1)-(A3), for each $n\geq0,i\in\ZZ$ and every simple constituent $M$ of $\pH^i K_n$, the support of $M$ is the whole of $\calB$.
\end{prop}
\begin{proof}
The property of $K_{n}$ to be proved is local for the \'etale topology of $\calB$, hence we may assume (A4) and (A5) hold. 

For $n\geq2g_{a}-1$, $h_n:\calH_n\to\cPic_n$ is a projective space bundle, the proposition is an easy consequence of Theorem \ref{th:supp}. In fact, the relative ample line bundle $\calO(1)$ along the fibers of $h_n$ gives a decomposition
\begin{equation*}
\bR h_{n,*}K_n\cong\bigoplus_{j=0}^{n-g_{a}}\Ql[-2j](-j).
\end{equation*}
Hence
\begin{equation*}
\pH^iK_n\cong\bigoplus_{j=0}^{n-g_{a}}\pH^{i-2j}L(-j).
\end{equation*}
In particular, any simple constituent $M$ of $\pH^{*}K_{n}$ is also a simple constituent of $\pH^{*}L$. By Theorem \ref{th:supp}, the support of $M$ is the whole of $\calB$.

Now we apply backward induction to $n$. Assuming the statement is true for any simple constituent of $\pH^{*}K_{n+1}$, we would like to deduce that the same is true for any simple constituent $M$ of $\pH^{*}K_{n}$. The idea is to show that $M$ appears as a direct summand of $\pH^{*}K_{n+1}$ via the map $T_{n}$. 

The sheaf $E_{1}=\nu_{*}\const{\calC}$ on $X\times\calB$ contains the constant sheaf as a direct summand. Fix a decomposition
\begin{equation*}
E_{1}=\const{X\times\calB}\oplus V
\end{equation*}
where $V$ is a sheaf on $X\times\calB$. Then we can write
\begin{equation*}
\Sym^{n}(E_{1})=\bigoplus_{i=0}^{n}V^{n}_{i}
\end{equation*}
such that $V^{n}_{i}|_{\tilU_{n}\times\calB}$ is the sum of direct summands of $E_{1}^{\boxtimes n}=(\Ql\oplus V)^{\boxtimes n}$ (under the binomial expansion) consisting of exactly $i$-factors of $V$. Let
\begin{equation*}
W^{n}_{i}=j_{n,!*}V^{n}_{i}
\end{equation*}
Then by Corollary \ref{c:midext},
\begin{equation}\label{binom}
E_{n}=\bigoplus_{i=0}^{n}W^{n}_{i}.
\end{equation}

We would like to understand the effect of the map $\tilT_{n}$ under the ``binomial expansion'' \eqref{binom}. Base change the diagram \eqref{T} to the generic point $(x_{1},\cdots,x_{n+1},\eta)\in \tilU_{n+1}\times\calB$, using the diagram \eqref{genericT}, we get
\begin{equation*}
\xymatrix{ &\bigsqcup_{i=1}^{n+1}\Gamma(p_{i})\ar[dl]_{\oleft{t}}\ar[dr]^{\oright{t}}\\
\bigsqcup_{i=1}^{n+1}\prod_{j\neq i}\nu^{-1}(x_{j},\eta)\times\{x_{i}\} & & \prod_{j=1}^{n+1}\nu^{-1}(x_{j},\eta)}
\end{equation*}
where $\Gamma(p_{i})$ is the graph of the natural projection:
\begin{equation*}
p_{i}:\prod_{j=1}^{n+1}\nu^{-1}(x_{j},\eta)\to\prod_{j\neq i}\nu^{-1}(x_{j},\eta)\times\{x_{i}\}.
\end{equation*}
This implies that the fiber of $\tilT_{n}$ at the point $(D=x_{1}+x_{2}+\cdots+x_{n+1},\eta)\in U_{n+1}\times\calB$ takes the form
\begin{equation*}
\tilT_{n}|_{(D,\eta)}=\bigoplus_{i=1}^{n+1}p_{i}^{*}:\bigoplus_{i=1}^{n+1}\bigotimes_{j\neq i}E_{1,x_{j}}\to\bigotimes_{j=1}^{n+1}E_{1,x_{j}}.
\end{equation*}
The pullback $p_{i}^{*}$ is the identity on the factor $E_{1,x_{j}}$ for $j\neq i$, and is the inclusion of the factor $\Ql$ into $E_{1,x_{i}}$ at the $i$-th factor. Using the expansion \eqref{binom}, we can rewrite $\tilT_{n}$ at $(D,\eta)$ as
\begin{equation*}
\bigoplus_{j=1}^{n}\varphi_{j}:\bigoplus_{j=1}^{n}\bigoplus_{i=1}^{n+1}V^{n}_{j,D-x_{i}}\to\bigoplus_{j=1}^{n}V^{n+1}_{j,D}\subset\bigoplus_{j=1}^{n+1}V^{n+1}_{j,D}
\end{equation*}
where the map $\varphi_{j}:\oplus_{i=1}^{n+1}V^{n}_{j,D-x_{i}}\to V^{n+1}_{j,D}$ can be understood the in the following way. In the following we omit the superscript of $V_{j,D}$ because it will be clear from the degree of $D$. By definition, we have
\begin{equation*}
V_{j,D}=\bigoplus_{D'\subset D,\deg(D')=j}\bigotimes_{x'\in D'}V_{x'},
\end{equation*}
hence
\begin{equation*}
\bigoplus_{i=1}^{n+1}V_{j,D-x_{i}}=\bigoplus_{D'\subset D,\deg(D')=j}\left(\bigoplus_{i,x_{i}\notin D'}\bigotimes_{x'\in D'}V_{x'}\right)=V_{j,D}^{\oplus n+1-j}.
\end{equation*}
The map $\varphi_{j}$ can be identified with
\begin{equation*}
\id^{\oplus n+1-j}:V_{j,D}^{\oplus n+1-j}\to V_{j,D}.
\end{equation*}

Both the source and the target of $\tilT_{n}$ are middle extension perverse sheaves (up to the shift $[d_{\calB}+n+1]$) from any open dense subset of $\Sym^{n+1}(X)\times\calB$ (the source being so because $\sigma_{n}$ is finite). The above calculation on the generic point implies
\begin{equation*}
\sigma_{n,*}(W^{n}_{j}\boxtimes\const{X})\cong (W^{n+1}_{j})^{\oplus n+1-j},
\end{equation*}
and the map $\tilT_{n}$ can be written as
\begin{equation*}
\bigoplus_{j=1}^{n}\id^{\oplus n+1-j}:\bigoplus_{j=1}^{n}(W^{n+1}_{j})^{\oplus n+1-j}\to\bigoplus_{j=1}^{n}W^{n+1}_{j}\subset\bigoplus_{j=1}^{n+1}W^{n+1}_{j}
\end{equation*}
In particular, every direct summand of $\sigma_{n,*}(E_{n}\boxtimes\Ql)$ appears as a direct summand of $E_{n+1}$. Applying $\bR f_{n+1,*}$, we conclude that all simple constituents of $\pH^{*}(K_{n}\otimes\cohog{*}{X})$ (which are the same as simple constituents of $\pH^{*}K_{n}$) appears in $\pH^{*}K_{n+1}$. This finishes the induction step.
\end{proof}

\begin{theorem}\label{th:pervmac}
Let $\pi:\calC\to\calB$ be a projective flat family of curves satisfying (A1)-(A3) in \S\ref{ss:assumption}. Let $f_n:\calH_n\to\calB$ be the relative Hilbert scheme of $n$-points on the fibers of $\pi$. Let $L^{i}$ be the perverse sheaves defined in \S\ref{ss:nosec} (which are the descent of the perverse cohomology sheaves of the family of compactified Jacobians). Then for any $n\geq0$ and $i\in\ZZ$, there is a canonical isomorphism
\begin{equation}\label{pervmac}
\pH^{i+d_{\calB}}\bR f_{n,*}\Ql\cong\bigoplus_{\max\{i-n,0\}\leq j\leq i/2}L^{i+d_{\calB}-2j}(-j).
\end{equation}
\end{theorem}
\begin{proof}
By Proposition \ref{p:supp} and Theorem \ref{th:supp}, both sides of \eqref{pervmac} are middle extension perverse sheaves on $\calB$. Therefore, it suffices to establish a canonical isomorphism in the form of \eqref{pervmac} on the open dense subset $\calB_{0}\subset\calB$ consisting of smooth fibers, or even at the generic point $\eta$ of $\calB$, i.e., a $\Gal(\overline{k(\eta)}/k(\eta))$-equivariant isomorphism
\begin{equation}\label{eta}
\cohog{i}{\Sym^n(\calC_{\bareta})}\cong\bigoplus_{\max\{i-n,0\}\leq j\leq i/2}L_{\bareta}^{i+d_{\calB}-2j}[-i-d_{\calB}](-j)
\end{equation}
where $\bareta$ is the geometric generic point above $\eta$. Since $L^{i+d_{\calB}}_{\bareta}[-i-d_{\calB}]$ is canonically $\cohog{i}{\Jac(\calC_{\bareta})}\cong\wedge^{i}\cohog{1}{\calC_{\bareta}}$, \eqref{eta} follows from the classical Macdonald formula \eqref{Mac} for $\calC_{\bareta}$ (which is $\Gal(\overline{k(\eta)}/k(\eta))$-equivariant and canonical). 
\end{proof}

\subsection{The perverse filtration}\label{ss:fil}
Suppose a family of curves $\calC\to\calB$ is such that $\calB$ is irreducible, and $\cJ=\cJac(\calC/\calB)$ is defined and smooth. For each geometric point $b\in\calB$, we get a {\em perverse filtration} on the total cohomology $\cohog{*}{\cJac(\calC_{b})}=\cohog{*}{\cJ_{b}}$. This is the increasing filtration defined as
\begin{equation*}
P_{\leq i}\cohog{*}{\cJ_{b}}:=(\leftexp{p}{\tau}_{\leq i+d_{\calB}}L)_{b}[-d_{\calB}].
\end{equation*}
Note that $\leftexp{p}{\tau}_{\leq i+d_{\calB}}L\to L$ is a direct summand by the decomposition \eqref{decompL}, hence $P_{\leq i}\cohog{*}{\cJ_{b}}$ defined above is indeed a subspace of $\cohog{*}{\cJ_{b}}$.

\begin{prop}\label{p:canonfil}
Let $\calC'\to\calB'$ be a family of curves satisfying (A1)-(A4), which is obtained from the family $\calC\to\calB$ by a base change $\varphi:\calB'\to\calB$. Assume $\calB$ is irreducible, $\cJ=\cJac(\calC/\calB)$ is defined and is smooth. Then for every geometric fiber $b'\in\calB'$ with image $b\in\calB$, the perverse filtration $P_{\leq i}$ on $\cohog{*}{\cJ_{b}}$ is the same as the perverse filtration $P'_{\leq i'}$ on $\cohog{*}{\cJ_{b'}}$ (under the identification $\cJ_{b}=\cJ_{b'}$). 
\end{prop}
\begin{proof}
Let $L\in D^{b}_{c}(\calB)$ and $L'\in D^{b}_{c}(\calB')$ be the direct image complex of $\cJ$ and $\cJ'$. The (non-canonical) decomposition \eqref{decompL} applies to both $L$ and $L'$. By proper base change, we have
\begin{equation*}
\bigoplus_{i}\varphi^{*}\pH^{i}L[-i]\cong\varphi^{*}L=L'\cong\bigoplus_{i}\pH^{i}L'[-i].
\end{equation*}
To prove the Proposition, it suffices to argue that $\varphi^{*}\pH^{i+\dim\calB}L\cong\pH^{i+\dim\calB'}L'$. Applying Proposition \ref{p:supp} to the family $\calC'\to\calB'$, every simple constituent of $\pH^{i}L'$ has support equal to $\calB'$, hence every simple constituent of $\pH^{j}\varphi^{*}\pH^{i}L'$ also has support $\calB'$.

Let $\calB'_{0}$ (resp. $\calB_{0}$) be the locus where the fiber curves are smooth, and let $\varphi_{0}:\calB_{0}'\to\calB_{0}$ be the restriction of $\varphi$. Then $\pH^{i+\dim\calB}L|_{\calB_{0}}$ and $\pH^{i+\dim\calB'}L'|_{\calB'_{0}}$ are both lisse and hence $\varphi_{0}^{*}(\pH^{i+\dim\calB}L|_{\calB_{0}})\cong\pH^{i+\dim\calB'}L'|_{\calB'_{0}}$. By the support property of $\varphi^{*}\pH^{i}L$ stated above, $\varphi^{*}\pH^{i}L$ must also be perverse, and it is the middle extension of its restriction to $\calB_{0}'$. In particular, we conclude $\varphi^{*}\pH^{i+\dim\calB}L\cong\pH^{i+\dim\calB'}L'$. This implies the Proposition.
\end{proof}

\subsection{Perverse filtration vs. Lefschetz filtration}\label{ss:pervLef}
We first recall the definition of the determinant line bundle of the family of compactified Jacobians $\cJ$. Let $\calF^{\univ}$ be the universal object over $\cJ\times_{\calB}\calC$. Let $\pr_{\cJ}:\cJ\times_{\calB}\calC\to\cJ$ be the projection. Let
\begin{equation*}
\calL_{\det}:=\det(\bR\pr_{\cJ,*}\calF^{\univ}),
\end{equation*}
which is a line bundle over $\cJ$. The iterated cup product by $c_{1}(\calL_{\det})$ induces a map:
\begin{equation}\label{c1i}
c_{1}(\calL_{\det})^{g_{a}-i}:\pH^{d_{\calB}+i}L\to\pH^{d_{\calB}+2g_{a}-i}L(g_{a}-i).
\end{equation}
Let $\calB_{0}\subset\calB$ be the locus where $\calC_{b}$ is smooth. It is well-known that $\calL_{\det}$ is ample when restricted to $\cJ_{b}$ for $b\in\calB_{0}$. Therefore, by the relative hard Lefschetz theorem \cite[Th\'eor\`eme 5.4.10, 6.2.10]{BBD}, the map $c_{1}(\calL_{\det})^{i}$ is an isomorphism over $\calB_{0}$, for $0\leq i\leq g_{a}$. By Proposition \ref{p:supp}, $\pH^{i}L$ is the middle extension of $\pH^{i}L|_{\calB_{0}}$ for any $i$, hence \eqref{c1i} is an isomorphism over the whole $\calB$. 

For each geometric point $b\in\calB$, we consider the Jacobson-Morozov filtration induced by the nilpotent action
\begin{equation*}
\cup c_{1}(\calL_{\det}):\cohog{*}{\cJ_{b}}\to\cohog{*}{\cJ_{b}}.
\end{equation*}
This is the unique increasing filtration $M_{\leq i}\cohog{*}{\cJ_{b}}$ such that $c_{1}(\calL_{\det})M_{\leq i}\subset M_{\leq i-2}$ and that $c_{1}(\calL_{\det})^{i}$ induces an isomorphism $\Gr^{M}_{i}\isom\Gr^{M}_{-i}$ (see \cite[Proposition 1.6.1]{WeilII}). We modify the filtration $M$ by setting
\begin{equation*}
F^{\geq i}\cohog{*}{\cJ_{b}}:=M_{\leq g_{a}-i}\cohog{*}{\cJ_{b}}.
\end{equation*} 
Then $F^{\geq i}$ is a decreasing filtration on $\cohog{*}{\cJ_{b}}$, which we call the {\em Lefschetz filtration.}

The fact that \eqref{c1i} is an isomorphism suggests a stronger statement, which we formulate as a conjecture.
\begin{conj}\label{conj:fil}
Assume (A1)-(A4) hold for $\calC\to\calB$. Then for every geometric point $b\in\calB$, the perverse filtrations $P_{\leq i}$ and the Lefschetz filtration $F^{\geq i}$ on $\cohog{*}{\cJ_{b}}$ are opposite to each other. 
\end{conj}

%   The following was  removed from the old introduction

%The compactified Jacobian carries a determinant line bundle $\calL_{\det}$ generalizing the $\Theta$-divisor for the usual Jacobian, which induces a nilpotent endomorphism
%\begin{equation*}
%\cup c_1(\calL_{\det}): \cohog{*}{\cJac(C)}\to \cohog{*}{\cJac(C)}.
%\end{equation*}
%Let $F^{\geq i}\cohog{*}{\cJac(C)}$ be the (modified) Jacobson-Morozov decreasing filtration associated to the nilpotent operator $\cup c_1(\calL_{\det})$ (see section \ref{ss:pervLef} for the definition). There is a unique way to normalize this filtration so that $\Gr^{i}_{F}\cohog{*}{\cJac(C)}$ is nonzero exactly for $0\leq i\leq 2g_{a}$, where $g_{a}$ is the arithmetic genus of $C$. We call this $F$ the {\em Lefschetz filtration}.

% applications

\section{Applications}

\subsection{Spectral curves}\label{ss:Hit} In this subsection, we give an example of a family $\pi:\calC\to\calB$ satisfying the conditions in (A1)-(A3) coming from the Hitchin fibration. 

Fix an integer $n\geq1$. Let $X$ be a smooth, projective and connected curve over $k$. Let $n$ be a positive integer and let $\calL$ be a line bundle over $X$. Let $\calA$ be the affine space
\begin{equation*}
\calA=\bigoplus_{i=1}^{n}\cohog{0}{X,\calL^{\otimes i}}.
\end{equation*}
viewed as an affine scheme over $k$.

Let $\Tot_{X}(\calL)=\underline{\Spec}_{X}(\oplus_{i\geq0}\calL^{\otimes-i}y^{i})$ be the total space of the line bundle $\calL$ ($y$ is a formal variable). We define a closed subscheme $\calY\subset\Tot_{X}(\calL)$ by the equation
\begin{equation*}
y^{n}+a_{1}y^{n-1}+\cdots+a_{n}=0, (a_{1},\cdots,a_{n})\in\calA.
\end{equation*}
Let $\pi:\calY\to\calA$ be the projection. This is the family of {\em spectral curves}. It appears in the study of Hitchin moduli space for the group $\GL_{n}$ (see Hitchin's original paper \cite[\S5.1]{Hitchin}). In particular, $\calA$ is the base of the Hitchin fibration.

Let $\Aint\subset\calA$ be the open locus where $\calY_{a}$ is integral (integrality is an open condition by \cite[Th\'eor\`eme 12.2.1]{EGA}). There is a stratification of $\Aint=\sqcup_{\delta\geq0}\Aint_{\delta}$ by the $\delta$-invariants of the spectral curves $\calY_{a}$.  Recall the following codimension estimate

\begin{lemma}\label{p:localdel}
\begin{enumerate}
\item []
\item {\em (Ng\^o \cite[p.4]{NgoDe})} If $\textup{char}(k)=0$, then $\codim_{\Aint}\Aint_{\delta}\geq\delta$ for all $\delta\geq0$. 
\item {\em (Ng\^o \cite[Proposition 5.7.2]{NgoFL}, which is based on a result of Goresky, Kottwitz and MacPherson \cite{GKM})} If $\textup{char}(k)>n$, then for each fixed $\delta_{0}\geq0$, there is an integer $N=N(\delta_{0})$ such that whenever $\deg(\calL)\geq N$ and $0\leq\delta\leq\delta_{0}$, we have
\begin{equation}\label{codim}
\codim_{\Aint}(\Aint_{\delta})\geq\delta.
\end{equation}
\end{enumerate}
\end{lemma}

If $\textup{char}(k)=0$, we take $\calB=\Aint$. If $\textup{char}(k)>n$, we fix $\delta_{0}\geq0$ and $\deg(\calL)\geq N(\delta_{0})$, and let $\calB=\sqcup_{\delta\leq\delta_{0}}\Aint_{\delta}\subset\Aint$ be the open locus where the estimate \eqref{codim} holds. We denote the restriction of $\pi$ to $\calB$ by the same symbol.

\begin{prop}
The family of curves $\pi:\calY\to\calB$ satisfies (A1)-(A3) in \S\ref{ss:assumption}. In particular, Theorem \ref{th:pervmac} applies to $\pi$.
\end{prop}
\begin{proof}
First of all, $\calY$ is closed in $\Tot_{X}(\calL)\times\calB$. Since we can compactify $\Tot_{X}(\calL)$ into a ruled surface over $X$ (hence projective) by adding a divisor at infinity, $\pi$ is a projective morphism. We check the conditions one by one.

(A1) Each fiber $\calY_{a}$ is integral because $\calB\subset\Aint$. Since $\calY_{a}\subset\Tot_{X}(\calL)$, it has planar singularities.

(A2) is proved by the second-named author in \cite[Claim 1 in the proof of Proposition 3.2.6]{YunJR}. 

(A3) is guaranteed by the choice of $\calB$.
\end{proof}

\subsection{Versal deformation of curves}
Let $C$ be an integral curve with planar singularities.  We construct in this subsection a family of curves $\pi:\calC\to\calB$ with $C=\pi^{-1}(b_{0})$ satisfying (A1)-(A3) in \S\ref{ss:assumption}. 

It follows from usual deformation-theoretic arguments (see \cite{Artin} and \cite{Vistoli}
for example) that $C$ can be included
in a family
\begin{equation}\label{versalfamily}\nonumber
\xymatrix{
C \ar@{^{(}->}[r]\ar[d]& \calC \ar[d]\\
b_0 \ar@{^{(}->}[r] & \calB \\
}
\end{equation}
that is versal at $b_{0} \in \calB$. More concretely, one can choose an
embedding of $C$ in $\PP^N$ for which 
\begin{equation*}
\cohog{1}{C,\calI_{Z}\otimes\calO_{\PP^{N}}(1)}=0
\end{equation*}
for all finite subschemes $Z\subset C$ with $\textup{length}(Z)\leq2g_{a}-1$, where $g_{a}$ is the arithmetic genus of $C$. By standard calculation, the above vanishing implies
\begin{equation}\label{van}
\cohog{1}{C,\calI_{Z}\otimes N_{C/\PP^{N}}}=0
\end{equation}
for any such $Z$. Let $\Hilb^{P}(\PP^N)$ be the Hilbert scheme of $\PP^{N}$ with Hilbert polynomial $P$ equal to that of $C$. Let $\pi:\calC\to\Hilb^{P}(\PP^{N})$ be the universal curve.

\begin{prop}\label{p:versal} Assume either $\textup{char}(k)=0$ or $\textup{char}(k)>\max\{\mult_{p}(C);p\in C\}$. Under the above choice of the embedding $C\hookrightarrow\PP^{N}$, there exists a Zariski neighborhood $\calB\subset\Hilb^{P}(\PP^{N})$ of $b_{0}=[C]$ over which the universal family $\pi:\calC\to\calB$ satisfies the conditions (A1)-(A3) in \S\ref{ss:assumption}.
\end{prop}
\begin{proof}
(A1) The property of being locally planar is open: consider the relative cotangent sheaf $\Omega_{\calC/\Hilb^{P}(\PP^{N})}$ and let $Z\subset\calC$ be the locus where its stalk has dimension at least 3. Clearly $Z$ is closed, hence $\pi(Z)\subset\Hilb^{P}(\PP^{N})$ is also closed. Then $\Hilb^{P}(\PP^{N})-\pi(Z)$ is precisely the locus where $\calC_{b}$ is locally planar. 
Also being integral is an open condition \cite[Th\'eor\`eme 12.2.1]{EGA}. Therefore we can take a Zariski neighborhood of $[C]$ in $\Hilb^{P}(\PP^{N})$ over which the fibers are integral and locally planar.

(A2) When $k=\CC$, this is proved by V.Shende in \cite[Proposition 13]{Shende}. In general, we proceed as follows. By \eqref{van}, $\cohog{1}{C,N_{C/\PP^{N}}}=0$, hence $[C]$ is in the smooth locus of $\Hilb^{P}(\PP^{N})$. Shrinking $\calB$ if necessary, we may assume $\calB$ is contained in the smooth locus of $\Hilb^{P}(\PP^{N})$. 

We first need the following lemma:
\begin{lemma}\label{l:hilbpn}
Any finite subscheme $Z'$ of length $n$ of a locally planar curve $C'\subset\PP^{N}$ lies in the smooth locus of $\Hilb^{n}(\PP^{N})$.
\end{lemma}
\begin{proof}
Since $Z'$ is planar, it lies in the closure of the locus of $n$ distinct
points on $\PP^N$, so the local dimension at $[Z']$ is at least $n\cdot N$,
and it suffices to bound the tangent space.  For this, we can assume that
$Z'$ is supported at a point and choose local coordinates so that $Z
\subset S = \Spec (k[[x,y]])\subset P = \Spec (k[[x,y, z_1, \dots, z_{N-2}]])$.
Since $I_{Z/P}$ is generated by the ideal $I_{Z/S}$ and $\{z_k\}$, we have
a surjection $$I_{Z/S}/I_{Z/S}^{2} \oplus \calO_{Z}^{\oplus(N-2)} \rightarrow
I_{Z/P}/I_{Z/P}^{2}\rightarrow 0$$ which leads to the inclusion
$$0\rightarrow \Hom(I_{Z/P}/I_{Z/P}^{2},\calO_Z)\rightarrow
\Hom(I_{Z/S}/I_{Z/S}^{2}, \calO_Z) \oplus \calO_{Z}^{N-2}.$$ By taking
lengths we get the desired upper bound of $n\cdot N$.
\end{proof}

Now fix $0\leq n\leq2g_{a}-1$. Let $\calI_{C^{\univ}}$ be the ideal sheaf of the universal curve $C^{\univ}\subset\calB\times\PP^{N}$. Let $\calO_{Z^{\univ}}$ be the structure sheaf of the universal subscheme $Z^{\univ}\subset\Hilb^{n}(\PP^{N})\times\PP^{N}$. Let $\calE$ be the complex $\bR\pr_{*}\bR\uHom(\calI_{C^{\univ}},\calO_{Z^{\univ}})$, where $\pr:\Hilb^{n}(\PP^{N})\times\calB\times\PP^{N}\to\Hilb^{n}(\PP^{N})\times\calB$ is the projection. Over $(Z'\subset C')\in\Hilb^{n}(\calC/\calB)$, $\Ext^{>0}_{\PP^{N}}(\calI_{C'},\calO_{Z'})=0$ for dimension reasons, hence $\calE$ is concentrated in degree zero in a neighborhood of $\Hilb^{n}(\calC/\calB)$ by semicontinuity, and is a vector bundle of rank $(N-1)n$ there. This vector bundle $\calE$ has a canonical section $s$ given by $\calI_{C^{\univ}}\hookrightarrow\calO_{\PP^{N}}\to\calO_{Z^{\univ}}$. Now $\Hilb^{n}(\calC/\calB)$ is cut off by the vanishing of $s$ on the smooth locus of $\Hilb^{n}(\PP^{N})\times\calB$, therefore the local dimension of $\Hilb^{n}(\calC/\calB)$ is at least
\begin{equation*}
\dim\Hilb^{n}(\PP^{N})^{\sm}+\dim\calB-\rank\calE=n+\dim\calB.
\end{equation*}

On the other hand, the tangent space of $\Hilb^{n}(\calC/\calB)$ at $(Z\subset C)$ is the kernel of the map
\begin{equation*}
\Hom(\calI_{Z}/\calI_{Z}^{2},\calO_{Z})\oplus\cohog{0}{C,N_{C/\PP^{N}}}\to \cohog{0}{Z,N_{C/\PP^{N}}|_{Z}}.
\end{equation*}
By \eqref{van}, this map is surjective. We argue in the proof of Lemma \ref{l:hilbpn} that $\Hom(\calI_{Z}/\calI_{Z}^{2},\calO_{Z})=Nn$. Clearly $h^{0}(C,N_{C/\PP^{N}})=\dim\calB$ and $h^{0}(Z,N_{C/\PP^{N}}|_{Z})=n\cdot\rank (N_{C/\PP^{N}})=n(N-1)$, therefore the dimension of the tangent space at $(Z\subset C)$ is $n+\dim\calB$. This together with the lower bound above gives the smoothness of $\Hilb^{n}(\calC/\calB)$ along the fiber $\Hilb^{n}(C)$. By openness of the smooth locus and properness of $\Hilb^{n}(\calC/\calB)\to\calB$, we may shrink $\calB$ further to ensure that $\Hilb^{n}(\calC/\calB)$ is smooth.

Above we fixed an integer $n$ and found a non-empty Zariski open subset of $\calB(n)\subset\calB$ over which $\Hilb^{n}(\calC/\calB)$ is smooth. Now $\cap_{n=0}^{2g_{a}-1}\calB(n)$ guarantees the smoothness condition (A2).

(A3) For every singularity $p\in C$, the deformation functor $\textup{Def}(\hatO_{C,p})$ has an {\em algebraic} miniversal hull $V_{p}$ by M.Artin's theorem \cite[Theorem 3.3]{Artin} and Elkik's theorem on isolated singularities \cite{Elkik}. More precisely, $V_{p}=\Spec R_{p}$ is of finite type over $k$ equipped with a point $0_{p}\in V_{p}(k)$ and an $R_{p}$-flat family of algebras $\hatO_{R_{p}}$ with an isomorphism $\hatO_{R_{p}}\otimes_{R_{p}}k(0_{p})\cong\hatO$. Let $\hatR_{p}$ be the completion of $R_{p}$ at $0_{p}$. The canonical morphism $v_{p}:\Spf\hatR_{p}\to\textup{Def}(\hatO_{C,p})$ is formally smooth and induces a bijection on the tangent spaces.

By the versality of $V_{p}$, there exists an \'etale neighborhood $\calB'$ of $[C]\in\calB$ and a pointed morphism
\begin{equation}\label{localDef}
(\calB',[C])\to\prod_{p\in C^{\sing}}(V_{p},0_{p}).
\end{equation}
Since $\calB$ is also versal, this morphism is smooth (see, for instance, \cite[Section A]{FGvS}, or 
\cite[Appendix A]{Laumon}). Therefore the codimension estimate holds for $\calB$ around $[C]$ with respect to the global $\delta$-invariant if and only if the codimension estimate holds for each $V_{p}$ around $0_{p}$ with respect to the local $\delta$-invariant (here we use the fact that the locus $V_{p}^{\sm}\subset V_{p}$ parametrizing smooth deformations is not empty).

In characteristic zero, the codimension estimate for $V_{p}$ is proven by Diaz and Harris \cite[Theorem 4.15]{DH}. In arbitrary characteristic, we do not know a reference but it can be deduced from the family of spectral curves, see Lemma \ref{l:codim} below.
\end{proof}

\begin{lemma}\label{l:codim}
Let $\hatO$ be the completed local ring of a planar curve singularity over $k$. Assume $\textup{char}(k)$ is either 0 or greater than the multiplicity of the singularity defined $\hatO$. Let $V$ be an algebraic miniversal deformation of $\hatO$, with the base point $0\in V(k)$ corresponding to $\hatO$. Then there is a Zariski neighborhood $V'$ of $0$ such that
\begin{equation*}
\codim_{V'}(V_{\delta}')\geq\delta.
\end{equation*}
for any $\delta$-constant stratum $V_{\delta}'\subset V'$.
\end{lemma}
\begin{proof}
By Weierstrass preparation theorem, we may choose a non-unit $0\neq t\in\hatO$ such that $\hatO\cong k[[t]][y]/(y^{n}+a_{1}^{*}(t)y^{n-1}+\cdots+a_{n}^{*}(t))$ for some $a_{i}^{*}(t)\in k[[t]]$, where $n$ is the multiplicity of $\hatO$. We may even assume $a_{i}^{*}(t)\in k[t]$ without changing the isomorphism type of $\hatO$, by a result of M.Artin and Hironaka \cite[Lemma 3.12]{AH}.

Let $Y^{*}=\Spec k[t,y]/(f)$ where $f(t,y)=y^{n}+a_{1}^{*}(t)y^{n-1}+\cdots+a_{n}^{*}(t)$. This is an affine plane curve with isolated singularities at $\{p_{1},\cdots,p_{r}\}$, and we may assume $\hatO_{Y^{*},p_{1}}\cong\hatO$. We have the miniversal deformations $(V_{i},0_{i})$ of the singularities $p_{i}$ as in the proof of Proposition \ref{p:versal}. Each $V_{i}$ is smooth at $0_{i}$ because $\textup{Def}(\hatO_{Y^{*},p_{i}})$ is formally smooth \cite[Th\'eor\`eme A.1.2(3)]{Laumon}.

Let $\calA_{N}$ be the Hitchin base associated with the curve $\PP^{1}$ and line bundle $\calO(N)$. We trivialize $\calO(N)$ over $\AA^{1}_{t}=\PP^{1}-\{\infty\}$, and identify $\calA_{N}$ with an affine space with coordinates $\una=(a_{i}(t))_{1\leq i\leq n}$ where $a_{i}(t)\in k[t]$ with $\deg(a_{i})\leq Ni$. Then $\calA_{N}$ parametrizes a family of affine spectral curves $\calY\to\calA_{N}$ with
\begin{equation*}
\calY_{\una}=\Spec k[t,y]/(y^{n}+a_{1}(t)y^{n-1}+a_{2}(t)y^{n-2}\cdots+a_{n}(t)), \textup{for }\una=(a_{i}(t))\in\calA_{N}.
\end{equation*}
We choose $N$ large enough so that the original data $(a_{i}^{*}(t))$ gives a point $\una^{*}\in\calA_{N}(k)$. By the versality of $V_{i}$, there is an \'etale neighborhood $\calA'_{N}$ of $\una^{*}\in\calA_{N}$ and a pointed morphism
\begin{equation*}
\rho:(\calA'_{N},\una^{*})\to\prod_{i=1}^{r}(V_{i},0_{i})
\end{equation*} 
such that the family of affine spectral curves $\calY\to\calA_{N}$ is, \'etale locally around the singularities of $Y^{*}=\calY_{\una^{*}}$, isomorphic to the pull-back of the disjoint union of the miniversal families over $V_{i}$.

We claim that $\rho$ is smooth at $\una^{*}$ for large $N$. Since both $\calA'_{N}$ and the $V_{i}$'s are smooth around the base points, we only need to show the tangent map of $\rho$ at $\una^{*}$ is surjective. We have a canonical isomorphism (see \cite[Part 1, \S4]{ArtinLecture})
\begin{equation*}
\oplus_{i=1}^{r}T_{0_{i}}V_{i}\cong k[t,y]/(f,\partial_{y}f,\partial_{t}f).
\end{equation*}
Identifying $T_{\una^{*}}\calA'_{N}$ with $\calA_{N}$ in the usual way, the tangent map $d\rho:T_{\una^{*}}\calA'_{N}\to \oplus_{i}T_{0_{i}}V_{i}$ takes the form
\begin{equation}\label{tang}
(a_{i}(t))_{1\leq i\leq n}\mapsto\sum_{i=1}^{n}a_{i}(t)y^{n-i}\in k[t,y]/(f,\partial_{y}f,\partial_{t}f).
\end{equation}
Since $f$ has only isolated singularities, $f$ does not have multiple factors. Since $\textup{char}(k)=0$ or 
$\textup{char}(k)>n$, $f$ and $\partial_{y}f$ are coprime as elements in $k(t)[y]$, hence the ideal $(f,\partial_{y}f)\subset k[t,y]$ contains some nonzero polynomial $g(t)\in k[t]$. Let $S=k[t]/(g(t))$, which is a finite-dimensional $k$-algebra. Then $\oplus_{i=1}^{r}T_{0_{i}}V_{i}$ is a quotient of $\sum_{i=0}^{n-1}S y^{i}$. For large $N$, the map $k[t]_{\deg\leq N}\to S$ is surjective, so is the tangent map \eqref{tang}. This proves $\rho$ is smooth at $\una^{*}$ for large $N$.

By the remark made after \eqref{localDef} in the proof of Proposition \ref{p:versal}, to prove the codimension estimate for $V_{1}$ around $0_{1}$, we only need to show that the same codimension estimate holds for $\calA_{N}$ around $\una^{*}$, for large enough $N$. To be precise, we would like to stratify $\calA_{N}$ by the $\delta$-invariants of the affine curves $\calY_{\una}$ (instead of the projective ones as considered in \S\ref{ss:Hit}). We call these strata $\calA_{N,\delta}$. For fixed $\delta_{0}$, we would like to show that once $N$ is large (depending on $\delta_{0}$), we have
\begin{equation*}
\codim_{\calA_{N}}(\calA_{N,\delta_{0}})\geq\delta_{0}.
\end{equation*}
The proof for this is completely analogous to Ng\^o's argument in \cite[Proposition 5.7.2]{NgoFL}, which works for spectral curves over any given curve, not necessarily complete. Our condition $\textup{char}(k)>n$ or $\textup{char}(k)=0$ is also needed here because the argument in \cite[Proposition 5.7.2]{NgoFL} relies on the codimension calculation of Goresky, Kottwitz and MacPherson in \cite{GKM}, which was done under the assumption $\textup{char}(k)>n$ or $\textup{char}(k)=0$. This completes the proof of the lemma.
\end{proof}

\subsection{Proof of Theorem \ref{th:main}}\label{ss:proof}
Let $C$ be an integral curve over $k$ with planar singularities. Above we constructed a family $\pi:\calC\to\calB$ containing $C=\pi^{-1}(b_{0})$ as a fiber, and satisfies (A1)-(A3) in \S\ref{ss:assumption}. Making an \'etale base change of $\calB$, we may further assume (A4) holds. Applying Theorem \ref{th:pervmac} to this family and taking the stalk of the relevant complexes at $b_{0}$, we obtain
\begin{equation*}
\cohog{*}{\Hilb^n(C)}\cong\bigoplus_{i+j\leq n, i,j\geq0}(\pH^{i+d_{\calB}}L)_{b_{0}}[-d_{\calB}-i-2j](-j).
\end{equation*}
As discussed in \S\ref{ss:fil}, the perverse filtration on $L$ induces a filtration $P_{\leq i}$ on $\cohog{*}{\cJac(C)}=L_{b_{0}}$, and $(\pH^{i+d_{\calB}}L)_{b_{0}}[-i-d_{\calB}]=\Gr^{P}_{i}(\cohog{*}{\cJac(C)})$. Now Theorem \ref{th:main} is almost proved, except we need to show that the perverse filtration $P_{\leq i}$ on $\cohog{*}{\cJac(C)}$ is independent of the choice of the deformation $\calC\to\calB$ satisfying (A1)-(A4). Let $\calC_{\calV}\to\calV$ be a versal deformation of $C$, which satisfies (A1)-(A3) as proved in Proposition \ref{p:versal}. Making an \'etale base change of $\calV$ (which preserves versality), we may assume $\calC_{\calV}\to\calV$ satisfies (A1)-(A4). By versality, there exists an \'etale neighborhood $\calB'$ of $b_{0}\in\calB$ and a morphism $\varphi:\calB'\to\calV$ such that $\calC'=\calC|_{\calB'}$ is obtained from $\calC_{\calV}$ via base change by $\varphi$. We thus have a diagram
\begin{equation*}
\calB\leftarrow\calB'\xrightarrow{\varphi}\calV.
\end{equation*}
Applying Proposition \ref{p:canonfil} to both arrows above, we conclude that $P_{\leq i}$ on $\cohog{*}{\cJac(C)}$ defined using the perverse filtration of the family $\calB$ is the same as the one defined using $\calV$. Since $\calV$ is fixed, $P_{\leq i}$ is independent of the choice of $\calB$. This shows the canonicity of the perverse filtration $P_{\leq i}$, and finishes the proof of Theorem \ref{th:main}.

\subsection{Functional equation}
Let
\begin{equation*}
Z_{\Hilb}(s,t,C):=\sum_{i,n\geq0}\dim\cohog{i}{\Hilb^{n}(C)}s^{i}t^{n}.
\end{equation*}
By Theorem \ref{th:main}, $Z_{\Hilb}(s,t,C)$ is a rational function of the form
\begin{equation*}
Z_{\Hilb}(s,t,C)=\frac{P(s,t)}{(1-t)(1-s^{2}t)}
\end{equation*}
where $P(s,t)=\sum_{i,j}\dim\Gr^{P}_{i}\cohog{j}{\cJac(C)}s^{j}t^{i}\in\ZZ[s,t]$ is a polynomial of bidegree degree $(2g_{a},2g_{a})$. Moreover, by the isomorphism \eqref{c1i}, the cup product by $c_{1}(\calL_{\det})$ gives an isomorphism
\begin{equation}\label{HLef}
c_{1}(\calL_{\det})^{i}:\Gr^{P}_{g_{a}-i}\cohog{j}{\cJac(C)}\isom\Gr^{P}_{g_{a}+i}\cohog{j+2i}{\cJac(C)}(i)
\end{equation}
for any $i,j\geq0$, which implies a symmetry on the polynomial $P(s,t)$, and hence the following functional equation for $Z_{\Hilb}(s,t,C)$:
\begin{equation*}
Z_{\Hilb}(s,t,C)=(st)^{2g_{a}-2}Z_{\Hilb}(s,s^{-2}t^{-1},C).
\end{equation*}

\subsection{Proof of Theorem \ref{th:zeta}}
The rationality of $Z_{\Hilb}(t,C/\FF_{q})$ follows by taking the Frobenius trace on the identity \eqref{Zeta}. The functional equation follows by taking the Frobenius trace on the isomorphism \eqref{HLef}.

\subsection{A local Macdonald formula}\label{ss:local}
Let $\hatO$ be a complete local reduced $k$-algebra of dimension 1 with maximal ideal $\frm$ and residue field $k$. We say $\hatO$ is {\em planar} if further more $\dim_{k}\frm/\frm^{2}\leq2$, i.e., $\hatO\cong k[[x,y]]/(f)$ for some $0\neq f\in k[[x,y]]$ without multiple factors. We may define the Hilbert scheme $\Hilb^{n}(\hatO)$ of length $n$ quotient algebras of $\hatO$. 

We may also define a ``compactified Picard'' for $\hatO$ as follows. Let $\hatK$ be the ring of fractions of $\hatO$, which is a finite product of $k((t))$. Let $\cPic(\hatO)$ be the functor which associates to every noetherian $k$-algebra $R$ the set of $R\hotimes_{k}\hatO$-submodules $M\subset R\hotimes_{k}\hatK$ such that for some (equivalently, any) non-zero-divisor $t\in\frm$, there exists an integer $i\geq0$ so that $R\hotimes_{k}t^{i}\hatO\subset M\subset R\hotimes_{k}t^{-i}\hatO$ and $R\hotimes_{k}t^{-i}\hatO/M$ (hence $M/R\hotimes_{k}t^{i}\hatO$) is a projective $R$-module. The functor $\cPic(\hatO)$ is a disjoint union $\cPic(\hatO)=\bigsqcup_{n\in\ZZ}\cPic^{n}(\hatO)$ according to the volume of $M$, defined as $\rk_{R}M/R\hotimes_{k}t^{i}\hatO-\dim_{k}\hatO/t^{i}\hatO$ for $i$ large. Each $\cPic^{n}(\hatO)$ is represented by an ind-scheme which is {\em not} of finite type. The reduced structure $\cPic^{\red}(\hatO)$ is an infinite union of projective varieties, all of dimension $\delta(\hatO)$ (the local $\delta$-invariant of $\hatO$). The above results are consequences of the theory of affine Springer fibers (see \cite{KL}, and also \cite[\S3]{NgoFL}), because $\hatO$ can be realized as the germ of a spectral curve.

Let $\Pic(\hatO)=\hatK^{\times}/\hatO^{\times}$, viewed as a group ind-scheme over $k$. Then $\Pic(\hatO)$ acts on $\cPic(\hatO)$: an element $g\in\hatK^{\times}/\hatO^{\times}$ sends $M\subset R\hotimes_{k}\hatK$ to $Mg$.

Let $\calO_{\hatK}$ be the ring of integers of $\hatK$. Then $\Lambda=\hatK^{\times}/\calO_{\hatK}^{\times}$ is a free abelian group of rank $r$ ($r$ is the number of analytic branches of $\hatO$). Let 
\begin{equation*}
\val:\hatK^{\times}\to\Lambda.
\end{equation*}
be the valuation map. Choosing a section of $\val$, we may identify $\Lambda$ as a subgroup of $\hatK^{\times}$, hence $\Lambda$ also acts on $\cPic(\hatO)$. It is proved by Kazhdan and Lusztig \cite{KL} that $\cPic^{\red}(\hatO)/\Lambda$ is a projective variety. Two different sections of $\val$ can be transformed to each other by an element of $\calO_{\hatK}$, which is a connected group scheme over $k$ acting also on $\cPic(\hatO)$, therefore the cohomology $\cohog{*}{\cPic^{\red}(\hatO)/\Lambda}$ is independent of the choice of the section of $\val$. 

%We still have a determinant line bundle $\calL_{\det}$ on $\cPic(\hatO)$: to each $M\subset R\hotimes_{k}\hatK$ it assigns the line
%\begin{equation*}
%\det_{R}(M/R\hotimes_{k}t^{i}\hatO)\otimes_{k}\det_{k}(\hatO/t^{i}\hatO)^{-1}.
%\end{equation*} 
%if $R\hotimes_{k}t^{i}\hatO\subset M$. This line bundle is $\hatK^{\times}$-equivariant hence descends to $\cPic^{\red}(\hatO)/\Lambda$. The nilpotent action of $\cup c_{1}(\calL_{\det})$ also induces a modified Jacobson-Morozov filtration $F^{\geq i}$ on $\cohog{*}{\cPic^{\red}(\hatO)/\Lambda}$, with $\Gr^{i}_{F}=0$ unless $0\leq i\leq2\delta(\hatO)$.

Let us recall the notion of the virtual Poincar\'e polynomial. Let $k$ be either $\overline{\FF_{p}}$ or a subfield of $\CC$. Let $X$ be a scheme of finite type over $k$. In both cases, we have a weight filtration $W_{\leq i}\cohog{*}{X}$ on the $\ell$-adic cohomology of $X$: when $k=\overline{\FF_{p}}$, we may assume $X=X_{0}\times_{\FF_{q}}\overline{\FF_{p}}$ for some $X_{0}$ over $\FF_{q}$, then the weight filtration comes from the absolute values of the $\Frob_{q}$ action on $\cohog{*}{X}$ (we need to fix an embedding $\iota:\Ql\hookrightarrow\CC$); when $k\subset\CC$, this comes from the weight filtration on the singular cohomology $\cohog{*}{X^{an},\QQ}$ ($X^{an}$ is the underlying analytic space of $X(\CC)$) and the comparison theorem between $\ell$-adic and singular cohomology. Then we define
\begin{equation*}
P^{\vir}(X,s)=\sum_{i,j}(-1)^{j}\dim \Gr^{W}_{i}\cohog{j}{X}s^{i}.
\end{equation*}
Similarly, suppose $\cohog{*}{X}$ carries a filtration $P_{\leq i}$ which is strictly compatible with the weight filtration, we can define the virtual Poincar\'e polynomial $P^{\vir}(\Gr^{P}_{i}\cohog{*}{X},s)$.

\begin{theorem}\label{th:local}
Let $\hatO$ be a planar complete local $k$-algebra of dimension 1, with $r$ analytic branches. Assume either $\textup{char}(k)=0$ or $\textup{char}(k)$ is greater than the multiplicity of the singularity defined by $\hatO$. Then there is filtration $P_{\leq i}$ on $\cohog{*}{\cPic^{\red}(\hatO)/\Lambda}$, normalized such that $\Gr^{P}_{i}=0$ unless $0\leq i\leq 2\delta(\hatO)$, strictly compatible with the weight filtration, such that we have an identity in $\ZZ[[s,t]]$
\begin{equation*}
\sum_{n}P^{\vir}(\Hilb^{n}(\hatO),s)t^{n}=\frac{\sum_{i}P^{\vir}(\Gr^{P}_{i}\cohog{*}{\cPic^{\red}(\hatO)/\Lambda},s)t^{i}}{(1-t)^{r}}.
\end{equation*}
\end{theorem}
\begin{proof}
Let us first remark on a product formula for Hilbert schemes of points. Let $C$ be a curve and $U=C-\{p_{1},\cdots,p_{m}\}$ be a Zariski open subset. Then we have the following identity in $\Var_{k}[[t]]$, where $\Var_{k}$ is the Grothendieck group of varieties over $k$.
\begin{equation*}
\left(\sum_{n\geq0}[\Hilb^{n}(U)]t^{n}\right)\prod_{i=1}^{m}\left(\sum_{n\geq0}[\Hilb^{n}(\hatO_{C,p_{i}})]t^{n}\right)=\sum_{n\geq0}[\Hilb^{n}(C)]t^{n}.
\end{equation*}

Given $\hatO$, there exists a rational projective curve $C$ over $k$, together with a point $p\in C(k)$ such that
\begin{itemize}
\item $C$ is nonsingular away from $p$;
\item $\hatO_{C,p}\cong\hatO$.
\end{itemize}
Let $\nu:\PP^{1}\to C$ be the normalization and let $\nu^{-1}(p)=\{p_{1},\cdots,p_{r}\}$, $U=C-\{p\}=\PP^{1}-\{p_{1},\cdots,p_{r}\}$. Applying the product formula to $C$ we get
\begin{equation*}
\left(\sum_{n\geq0}[\Hilb^{n}(U)]t^{n}\right)\left(\sum_{n\geq0}[\Hilb^{n}(\hatO)]t^{n}\right)=\sum_{n\geq0}[\Hilb^{n}(C)]t^{n}.
\end{equation*}
Applying the product formula to $\PP^{1}$ we get
\begin{eqnarray*}
\frac{\sum_{n\geq0}[\Hilb^{n}(U)]t^{n}}{(1-t)^{r}}&=&\left(\sum_{n\geq0}[\Hilb^{n}(U)]t^{n}\right)\prod_{i=1}^{r}\left(\sum_{n\geq0}[\Hilb^{n}(\hatO_{\PP^{1},p_{i}})]t^{n}\right)\\
&=&\sum_{n\geq0}[\Hilb^{n}(\PP^{1})]t^{n}=\sum_{n\geq0}[\PP^{n}]t^{n}=\frac{1}{(1-t)(1-\LL t)}.
\end{eqnarray*}
where $\LL$ is the class of $\AA^{1}$. Taking the quotient of the two identities, we get
\begin{equation*}
\sum_{n\geq0}[\Hilb^{n}(\hatO)]t^{n}=\frac{1-\LL t}{(1-t)^{r-1}}\sum_{n\geq0}[\Hilb^{n}(C)]t^{n}.
\end{equation*}
Taking the virtual Poincar\'e polynomials of both sides, and applying Theorem \ref{th:main} to $C$ (which is applicable by the assumption on $\textup{char}(k)$), we get
\begin{equation}\label{pre}
\sum_{n\geq0}P^{\vir}(\Hilb^{n}(\hatO),s)t^{n}=\frac{\sum_{i}P^{\vir}(\Gr^{P}_{i}\cohog{*}{\cJac(C)},s)t^{i}}{(1-t)^{r}}.
\end{equation}

Analogous to Ng\^o's product formula \cite[\S4.15]{NgoFL}, we also have a product formula relating the local and global Picard schemes, which we state below (the proof is the same as \cite[Th\'eor\`eme 4.6]{NgoFib}). We have natural morphisms $\Pic(\hatO)\to\Pic(C)$ and $\cPic(\hatO)\to\cPic(C)$ given by gluing the local objects with the trivial line bundle on $U$. There is a homeomorphism
\begin{equation}\label{prod}
\cPic^{\red}(\hatO)\twtimes{\Pic^{\red}(\hatO)}\Pic(C)\to\cPic(C).
\end{equation}
The local and global Picard fit into a commutative diagram of exact sequences
\begin{eqnarray*}
\xymatrix{1\ar[r] & \calO^{\times}_{\hatK}/\hatO^{\times}\ar[r]\ar@{=}[d] & \Pic(\hatO)\ar[r]^{\val}\ar[d] & \Lambda\ar[d]^{\text{sum}}\ar[r] & 1\\
1\ar[r] &\calO^{\times}_{\hatK}/\hatO^{\times}\ar[r] &\Pic(C)\ar[r]^{\nu^{*}} & \Pic(\PP^{1})\ar[r] & 1}
\end{eqnarray*}
Since $\Pic(\PP^{1})=\ZZ$, the map ``sum'' just means the sum map $\Lambda\cong\ZZ^{r}\to\ZZ$. Therefore, after choosing a section of $\val$, we have homeomorphisms
\begin{equation}\label{lgPic}
\cPic(C)\cong\cPic^{\red}(\hatO)/\ker(\textup{sum}) \textup{ and }\cJac(C)\cong\cPic^{\red}(\hatO)/\Lambda.
\end{equation}
Finally we define the filtration $P_{\leq i}$ on $\cohog{*}{\cPic^{\red}(\hatO)/\Lambda}$ to be the transport of the perverse filtration on $\cohog{*}{\cJac(C)}$. It is easy to see that this filtration satisfies the requirements in the Theorem. Therefore we get an isomorphism of bi-filtered graded vector spaces
\begin{equation*}
(\cohog{*}{\cJac(C)},P_{*},W_{*})\cong(\cohog{*}{\cPic^{\red}(\hatO)/\Lambda},P_{*},W_{*}).
\end{equation*}
This, together with \eqref{pre}, implies the desired formula.
\end{proof}

\end{document}